\documentclass[a4paper,11pt]{article}

\usepackage[dvipdfmx]{graphicx}      
\usepackage{amsmath}                 
\usepackage{amssymb}                 
\usepackage{amsthm}                  
\usepackage{bm}                      
\usepackage{tikz}

\newtheorem{theorem}{Theorem}

\newtheorem{fact}{Fact}
\newtheorem{lem}{Lemma}
\newtheorem{prop}{Proposition}

\newtheorem{remark}{Remark}

\begin{document}

\title{Structural similarity between polyhedral embeddings and their duals and its application to self-duality of pathwidth} 
\author{Hikaru Yokoi\thanks{School of Fundamental Science and Technology, Graduate School of Science and Technology, Keio University, Yokohama, Japan. email : hikaru.yokoi2@gmail.com}\\ 
Keio University}
\maketitle

\begin{abstract}
Let $G$ be a graph embedded on a closed surface. 
We call $G$ a \emph{polyhedral embedding} if all facial walks are cycles, and any two of them are either disjoint or intersect in a single vertex or a single edge. 
In this paper, we present a new bound on the relation between the pathwidth of a polyhedral embedding and its dual. 
More precisely, we prove that for a polyhedral embedding $G$ on a closed surface with Euler characteristic $\chi$, $\mathsf{pw}(G^*) \leq 3\ \mathsf{pw}(G)+c$, where $c$ is a constant depending only on $\chi$. 
This result improves the coefficient of $\mathsf{pw}(G)$ in the previously known bound by Fomin and Thilikos (2007) and extends that of Amini, Huc, and P\'erennes (2009) for plane graphs. 
Furthermore, we obtain analogous bounds on the treewidth and pathwidth of the face subdivision of a polyhedral embedding. 
Our approach is based on a new quantitative estimate which demonstrates the structural similarity between a polyhedral embedding and its dual.
\\

\end{abstract}

\section{Introduction}
In this paper, we consider only simple, undirected, and finite graphs. 
We follow the terminology of \cite{D2017}. 
A \emph{closed surface} is a compact, connected, and two-dimensional manifold without boundary. 
Let $G$ be a graph embedded in a closed surface. 
We denote the geometric dual of $G$ by $G^*$. 
A \emph{face subdivision} of $G$, denoted by $G^{\mathsf{fs}}$, is the graph obtained from $G$ by adding a vertex to each face and connecting it to all the vertices on the boundary.
We say that $G$ is a \emph{polyhedral embedding} if all facial walks are cycles, and any two of them are either disjoint or intersect in a single vertex or a single edge. 
(For another equivalent definition, see \cite{MT2001}.) 
A polyhedral embedding is a natural generalization of 3-connected plane graphs. 
The following fact on polyhedral embeddings is well-known.

\begin{fact} [\cite{MT2001}] \label{fact1}
Let $G$ be a polyhedral embedding in a closed surface. 
Then $G^*$ is also polyhedral. 
\end{fact}

Diestel and K\"uhn \cite{DK2005} introduced an $H$-decomposition of graphs as a generalization of tree-decompositions. 
(Precisely, they call the following decomposition a \emph{connected} $H$-decomposition.) 
Let $G$ and $H$ be graphs. 
An \emph{$H$-decomposition} of $G$ is a pair $(H,\mathcal{B})$, where $\mathcal{B}=(G_h)_{h\in V(H)}$ is a family of subsets of $V(G)$ (called \emph{bags}) indexed by the vertices of $H$ satisfying the following conditions. 
\begin{itemize}
\item[(D1)] Every vertex of $G$ is contained in $G_h$ for some $h\in V(H)$. 
\item[(D2)] For every edge $gg'$ of $G$, there exists a vertex $h$ of $H$ such that $\{g, g'\}\subseteq G_h$, or an edge $hh'$ of $H$ such that $g\in G_h$ and $g'\in G_{h'}$.
\item[(D3)] For every vertex $g$ of $G$, the set $\{ h\in V(H) \mid g\in G_h \}$ induces a connected subgraph of $H$.
\end{itemize} 


We define the \emph{width of $(H,\mathcal{B})$} as the maximum order of $G_h$ among $h\in V(H)$. 
For a family of graphs $\mathcal{H}$, the \emph{$\mathcal{H} \textnormal{-} \mathsf{width}$} of $G$, denoted by $\mathcal{H} \textnormal{-} \mathsf{wd}(G)$, is the minimum width of all $H$-decompositions of $G$ with $H\in \mathcal{H}$. 
In particular, if $\mathcal{H}$ consists of a single graph $H$, we simply write $\mathcal{H} \textnormal{-} \mathsf{wd}(G)$ as $H \textnormal{-} \mathsf{wd}(G)$.

To state the standard definitions of a tree-decomposition and a path-decomposition of graphs, we repeat some parts of the definition of $H$-decompositions. 
Let $G$ and $H$ be graphs. 
An $H$-decomposition $(H,\mathcal{B})$ of $G$ is called a \emph{tree-decomposition} (resp. \emph{path-decomposition}) of $G$ if $H$ is a tree (resp. path) and $(H,\mathcal{B})$ satisfies (D1), (D3), and the following condition (D2'), which is stronger than (D2). 

\begin{itemize}
\item[(D2')] For every edge $gg'$ of $G$, there exists a vertex $h$ of $H$ such that $\{g, g'\}\subseteq G_h$,
\end{itemize}
where $\mathcal{B}=(G_h)_{h\in V(H)}$. 
The \emph{width} of a tree-decomposition (resp. path-decomposition) $(H,\mathcal{B})$ is $\max_{h\in V(H)} |G_h|- 1$. 
The \emph{treewidth} (resp. \emph{pathwidth}) of $G$, denoted by $\mathsf{tw}(G)$ (resp. $\mathsf{pw}(G)$), is the minimum width of all tree-decompositions (resp. path-decompositions) of $G$.

Given a graph $G$ embedded on a closed surface, how similar are $G$ and its dual $G^*$ in terms of width parameters? 
Research motivated by this question, so-called \emph{self-duality of width parameters}, originated from the seminal work \cite{RS1984} in the \emph{Graph Minors Project}. 
For treewidth and branchwidth, these self-duality problems were completely solved; the gaps between the widths of $G$ and $G^*$ are bounded by a constant depending only on Euler genus, for both these two parameters \cite{KLPSZ2024, M2012}. 

On the other hand, although the self-duality of pathwidth has been investigated for several classes of graphs \cite{AHP2009, BF2002, CHS2007, FT2007}, a tight bound has been obtained only for 2-connected outerplanar graphs \cite{CHS2007}. 
Note that the assumption of 2-connectivity is natural, since the pathwidth of trees is unbounded. 
In this paper, we focus on polyhedral embeddings. 
For this class of graphs, the following bounds are known. 

\begin{theorem} [\cite{FT2007}]\label{thm1}
Let $G$ be a polyhedral embedding on an orientable surface with genus $g$. 
Then, $\mathsf{pw}(G^*)\leq 6\ \mathsf{pw}(G)+6g-2$.
\end{theorem}

\begin{theorem} [\cite{AHP2009}]\label{thm2}
Let $G$ be a 3-connected plane graph. 
Then, $\mathsf{pw}(G^*)\leq 3\ \mathsf{pw}(G)+2$.
\end{theorem}

Fomin and Thilikos \cite{FT2007} constructed an infinite family of 3-connected planar graphs which shows that the coefficient of $\mathsf{pw}(G)$ cannot be lowered below 1.5. 
Inspired by these results, we prove that the coefficient 3 in Theorem~\ref{thm2} can also be achieved for polyhedral embeddings on any closed surface and for general $\mathcal{H} \textnormal{-} \mathsf{width}$. 
Our proof method is a slight modification of that used in \cite{AHP2009}, yet provides a new perspective on the structural similarity between a polyhedral embedding and its dual (see Theorem~\ref{thm6} in Section~\ref{sec2}).

\begin{theorem}\label{thm3}
Let $G$ be a polyhedral embedding on a closed surface with Euler characteristic $\chi$. 
Let $H$ be a graph. 
Then, \begin{equation*} 
H \textnormal{-} \mathsf{wd}(G^*) \leq \begin{cases}
3\ H \textnormal{-} \mathsf{wd}(G) & \text{if}\ \chi\in \{1,2\}, \\
3\ H \textnormal{-} \mathsf{wd}(G)+2 & \text{if}\ \chi=0, \\
3\ H \textnormal{-} \mathsf{wd}(G)-4\chi + 1 & \text{if}\ \chi<0.
\end{cases}
\end{equation*}
\end{theorem}

\begin{theorem}\label{thm4}
Let $G$ be a polyhedral embedding on a closed surface with Euler characteristic $\chi$. 
Then, \begin{equation*} 
\mathsf{pw}(G^*) \leq \begin{cases}
3\ \mathsf{pw}(G)+2 & \text{if}\ \chi = 1, \\
3\ \mathsf{pw}(G)+4 & \text{if}\ \chi=0, \\
3\ \mathsf{pw}(G)-4\chi + 3 & \text{if}\ \chi<0.
\end{cases}
\end{equation*}
\end{theorem}


In addition, we obtain the following bounds on the treewidth and pathwidth of the face subdivision of a polyhedral embedding.

\begin{theorem}\label{thm5}
Let $G$ be a polyhedral embedding on a closed surface with Euler characteristic $\chi$.
Then, \begin{equation*} 
\mathsf{p}(G^{\mathsf{fs}}) \leq \begin{cases}
4\ \mathsf{p}(G)+3 & \text{if}\ \chi\in \{1,2\}, \\
4\ \mathsf{p}(G)+5 & \text{if}\ \chi=0, \\
4\ \mathsf{p}(G)-4\chi + 4 & \text{if}\ \chi<0,
\end{cases}
\end{equation*}
where $\mathsf{p}(\cdot)$ denotes either the treewidth or the pathwidth of a graph.
\end{theorem}



\section{Proofs of Theorems~\ref{thm3} and \ref{thm4}}\label{sec2}

Let $G$ be a graph. 
For a vertex $v$ of $G$, we write the degree of $v$ in $G$ as $d_G(v)$. 
For a vertex set $X$ of $G$, $G[X]$ denotes the subgraph of $G$ induced by $X$. 
We start with the following lemma. 

\begin{lem}\label{lem1}
Let $F$, $G$, and $H$ be graphs. 
Let $(G,\mathcal{B}_1)$ and $(H,\mathcal{B}_2)$ be a $G$-decomposition of $F$ of width $G \textnormal{-} \mathsf{wd}(F)$ and an $H$-decomposition of $G$ of width $H \textnormal{-} \mathsf{wd}(G)$, respectively, where $\mathcal{B}_1 = (F_g)_{g\in V(G)}$ and $\mathcal{B}_2 = (G_h)_{h\in V(H)}$. 
For each $h\in V(H)$, we define $\tilde{F}_h :=\bigcup_{g\in G_h} F_g$. 
Then the following statements hold. 

\begin{itemize}
\item[(a)] The pair $(H,(\tilde{F}_h)_{h\in V(H)})$ is an $H$-decomposition of $F$.
\item[(b)] For each $h\in V(H)$, $|\tilde{F}_h|\leq G \textnormal{-} \mathsf{wd}(F) \cdot H \textnormal{-} \mathsf{wd}(G)$.
\end{itemize}
\end{lem}

\begin{proof}
The statement (b) is immediate, since for each $h\in V(H)$, we have 
\[ |\tilde{F}_h| = \left| \bigcup_{g\in G_h} F_g \right|\leq \sum_{g\in G_h} |F_g| \leq |G_h|\cdot G \textnormal{-} \mathsf{wd}(F) \leq G \textnormal{-} \mathsf{wd}(F) \cdot H \textnormal{-} \mathsf{wd}(G). \]
So we confirm the statement (a). 
\begin{itemize}
\item[(D1)] For a vertex $f$ of $F$, take a vertex $g$ of $G$ with $f\in F_g$ and a vertex $h$ of $H$ with $g\in G_h$. 
Then $f\in \tilde{F}_h$.
\item[(D2)] Let $ff'$ be an edge of $F$. 
Then, at least one of the following two cases occurs.
\begin{itemize}
\item[Case 1] : Both $f$ and $f'$ are contained in some bag $F_g$.
\item[Case 2] : There is an edge $gg'$ of $G$ such that $f\in F_g$ and $f'\in F_{g'}$. 
\end{itemize}
In Case 1, we have $g\in G_h$ for some $h\in V(H)$ and thus $\{f,f'\}\subseteq \tilde{F}_h$. 
In Case 2, if $\{g,g'\}\subseteq G_h$ for some $h\in V(H)$, then $\{f,f'\}\subseteq \tilde{F}_h$. 
Otherwise there is an edge $hh'$ of $H$ such that $g\in G_h$ and $g'\in G_{h'}$, which implies that $f\in \tilde{F}_h$ and $f'\in \tilde{F}_{h'}$. 
\item[(D3)] Let $f$ be a vertex of $F$. 
Put $A:=\{g\in V(G) \mid f\in F_g\}$ and $B:=\{h\in V(H) \mid f\in \tilde{F}_h\}$. 
By definition, $B=\{h\in V(H) \mid g\in G_h \text{ for some }  g\in A\}$. 
We show that $B$ induces a connected subgraph of $H$. 
Note that
\smallskip\\
\noindent\hspace*{\fill} 
(*) : for each $g\in A$, $H[\{ h\in V(H) \mid g\in G_h \}]$ is connected.
\hspace*{\fill} 
\smallskip\\
Since $G[A]$ is connected, it suffices to show that $H[\{ h\in V(H) \mid g\in G_h \text{ or } g'\in G_h \}]$ is connected for every $gg'\in E(G[A])$. 
For any $gg'\in E(G[A])$, because there is $h\in V(H)$ such that $\{g,g'\}\subseteq G_h$, or $hh'\in E(H)$ such that $g\in G_h$ and $g'\in G_{h'}$, this assertion holds by (*). 
\end{itemize}
\end{proof}

\begin{remark}\label{rem1}
By a similar argument as in the proof of Lemma~\ref{lem1}, we can prove the following. 
In Lemma~\ref{lem1}, if $(H,\mathcal{B}_2)$ is a tree-decomposition (resp. path-decomposition) of $G$, then $(H,(\tilde{F}_h)_{h\in V(H)})$ is also a tree-decomposition (resp. path-decomposition) of $F$. 
Furthermore, the statement (b) holds if we replace $H \textnormal{-} \mathsf{wd}(G)$ by the width of  $(H,\mathcal{B}_2)$ plus 1 as a tree-decomposition (resp. a path-decomposition) of $G$. 
\end{remark}

Throughout the rest of this section, let $G$ be a polyhedral embedding on a closed surface with Euler characteristic $\chi$, unless otherwise stated. 
The following theorem plays a key role in the proof of Theorems~\ref{thm3} and~\ref{thm4}.

\begin{theorem}\label{thm6}
The following statements hold.
\begin{enumerate}
\item[(1)] If $\chi\in \{1,2\}$, then $G \textnormal{-} \mathsf{wd}(G^*)\leq 3$. 
\item[(2)] If $\chi=0$, then there exists $S\subseteq V(G^*)$ with $|S|\leq 2$ such that $G \textnormal{-} \mathsf{wd}(G^*-S)\leq 3$.
\item[(3)] If $\chi<0$, then there exists $S\subseteq V(G^*)$ with $|S|\leq -4\chi+1$ such that $G \textnormal{-} \mathsf{wd}(G^*-S)\leq 3$.
\end{enumerate}
\end{theorem}

What does Theorem~\ref{thm6} imply? 
Theorem~\ref{thm6} states that, after removing a small number of ``apex vertices" from the dual of a polyhedral embedding, the resulting graph has a structure ``similar" to that of the original graph.
Since Fact~\ref{fact1} suggests that a polyhedral embedding and its dual have similar structures, Theorem~\ref{thm6} can be viewed as supporting  Fact~\ref{fact1} from the perspective of a width parameter. 
We first see that Theorem~\ref{thm6} implies Theorems~\ref{thm3} and~\ref{thm4}. 

\begin{proof}[Proof of Theorem~\ref{thm3}] 
For $\chi\in \{1,2\}$, it is easy to see that $H \textnormal{-} \mathsf{wd}(G^*) \leq 3\ H \textnormal{-} \mathsf{wd}(G)$ follows from (1) and Lemma~\ref{lem1}. 
So, we consider the case $\chi\leq 0$.
Let $S\subseteq V(G^*)$ be the vertex set in Theorem~\ref{thm6}. 
By Lemma~\ref{lem1}, $G^*-S$ admits an $H$-decomposition $(H,\mathcal{B})$ of width at most $3\ H \textnormal{-} \mathsf{wd}(G)$. 
Adding all vertices of $S$ to each bag of $(H,\mathcal{B})$, by (2) and (3), we obtain an $H$-decomposition of $G^*$ of width at most 
\begin{equation*} 
3\ H \textnormal{-} \mathsf{wd}(G)+|S| \leq 
\begin{cases}
3\ H \textnormal{-} \mathsf{wd}(G)+2 & \text{if}\ \chi= 0, \\
3\ H \textnormal{-} \mathsf{wd}(G)-4\chi + 1 & \text{if}\ \chi<0,
\end{cases}
\end{equation*}
as desired.
\end{proof}

\begin{proof}[Proof of Theorem~\ref{thm4}] 
Let $(P_n,\mathcal{X})$ be a path-decomposition of $G$ of width $\mathsf{pw}(G)$. 
Let $H$ be isomorphic to $P_n$. 
By definition, $H \textnormal{-} \mathsf{wd}(G) = \mathsf{pw}(G)+1$ and $\mathsf{pw}(G^*)+1\leq H \textnormal{-} \mathsf{wd}(G^*)$.
Recall Remark~\ref{rem1}. 
By repeating the same argument in the proof of Theorem~\ref{thm3} assuming Theorem~\ref{thm6}, we have the following inequalities. \begin{equation*} 
\mathsf{pw}(G^*)+1 \leq 
\begin{cases}
3\ (\mathsf{pw}(G)+1) & \text{if}\ \chi= 1, \\
3\ (\mathsf{pw}(G)+1)+2 & \text{if}\ \chi= 0, \\
3\ (\mathsf{pw}(G)+1)-4\chi+1 & \text{if}\ \chi<0.
\end{cases}
\end{equation*}
\end{proof}

We move on to a proof of Theorem~\ref{thm6}. 
Let $\tau$ be a map from $V(G^*)$ to $E(G)$. 
We say that $\tau$ is an \emph{edge-assignment} of $G$ if $\tau$ is an injection and $\tau(f)$ lies in the boundary of the face $f$ for every $f\in V(G^*)$. 
For each $v\in V(G)$, we let 
\[ G^*_v := \{ f\in V(G^*) \mid v \text{ is on the face $f$ but not an endpoint of } \tau(f) \}. \]
Then we have the following proposition.

\begin{prop}\label{prop1}
The pair $(G,(G^*_v)_{v\in V(G)})$ is a $G$-decomposition of $G^*$.
\end{prop}
 
\begin{proof}
Let $f$ be a vertex of $G^*$. 
Since the set $\{ v\in V(G) \mid f\in G^*_v \}$ consists of the vertices on $f$ which are not endpoints of $\tau(f)$, it induces a nonempty subpath, in particular, a connected subgraph of $G$. 
Thus, $(G,(G^*_v)_{v\in V(G)})$ satisfies the conditions (D1) and (D3). 

Now we confirm that it also satisfies the condition (D2). 
Let $ff'$ be an edge of $G^*$. 
Consider the (unique) edge $vv'$ such that the faces $f$ and $f'$ share in $G$. 
If at least one of $v$ and $v'$ is neither an endpoint of $\tau(f)$ nor of $\tau(f')$, then, by symmetry, we may assume that $v$ is such a vertex. 
In this case, we have $f, f'\in G^*_v$. 
Otherwise, since $\tau$ is injective, either $v$ or $v'$ is not an endpoint of $\tau(f)$ or $\tau(f')$. 
By symmetry, assume that $v$ is such a vertex; in particular, $v$ is not an endpoint of $\tau(f)$ but is an endpoint of $\tau(f')$. 
Then, we can find a neighbor $w$ of $v$ which lies on the face $f'$ but is not an endpoint of $\tau(f')$. 
Since $f\in G^*_{v}$ and $f'\in G^*_{w}$, the edge $vw$ is a desired one. 
\end{proof}

For an edge-assignment $\tau$ of $G$, let $H_{\tau} := G-\{ \tau(f) \mid f\in V(G^*) \}$.  

\begin{prop}\label{prop2}
For every $v\in V(G)$, $|G^*_v|=d_{H_{\tau}}(v)$.
\end{prop}

\begin{proof}
Let $v$ be a vertex of $G$. 
Since $G$ is a polyhedral embedding, $v$ is incident with $d_G(v)$ distinct faces.
Thus, the number of faces incident with $v$ equals $d_G(v)$. 
Let $X$ be the set of faces $f$ such that $v$ is an endpoint of $\tau(f)$. 
By definition, we have $|G^*_v|= d_G(v)-|X|$. 
Because $\tau$ is an injection, this value equals $d_{H_{\tau}}(v)$, as desired.
\end{proof}

In Propositions~\ref{prop1} and~\ref{prop2}, we show that a $G$-decomposition of $G^*$ can be obtained from an edge-assignment of $G$ and observe its property. 
Next, we present a construction of an edge-assignment of $G$. 

\begin{lem}\label{lem2}
Let $G$ be a polyhedral embedding on a closed surface with Euler characteristic $\chi\leq 1$. 
For any spanning tree $T$ of $G$, there exists an edge-assignment $\tau$ from $V(G^*)$ to $E(G)\setminus E(T)$. 
In particular, for the graph $H_{\tau}$ defined by such an edge-assignment $\tau$, the following statements hold.
\begin{itemize}
\item [(a)] $T\subseteq H_{\tau}$, and
\item [(b)] $|E(H_{\tau})| = |E(T)|-\chi+1$.
\end{itemize}
\end{lem}

\begin{proof}
Fix a spanning tree $T$ of $G$. 
We construct the graph $T'$ from $T$ as follows. 
The vertex set of $T'$ is $V(G^*)$, and two vertices $f$ and $f'$ of $T'$ are adjacent if and only if they share an edge of $E(G)\setminus E(T)$. 
By the construction of $T'$, $T'$ is connected. 
In addition, by Euler's formula, we have 
\[ |E(T')| = |E(G)| - |E(T)| = |E(G)| - |V(G)| + 1 = |V(G^*)| - \chi + 1\geq |V(T')|. \]
Hence, $T'$ contains a connected unicyclic spanning subgraph $T''$. 
Note that $V(T'')=V(G^*)$. 
Then, we can easily take a bijection $\sigma$ from $V(T'')$ to $E(T'')$ such that $\sigma(f)$ is incident with $f$ for each vertex $f$ of $T''$. 
By the construction of $\sigma$, $\sigma(f)$ lies in the face $f$ in $G$. 
Thus, we obtain an edge-assignment $\tau$ from $V(G^*)$ to $E(G)\setminus E(T)$ via $\sigma$, as desired.

Let $H_{\tau}$ be the graph defined by the above edge-assignment $\tau$. 
Then, (a) is trivial, since the image of $\tau$ is contained in $E(G)\setminus E(T)$. 
It is also easy to see that (b) holds by Euler's formula as follows.  
\[ |E(H_{\tau})| = |E(G)| - |V(G^*)| = |V(G)| - \chi = |E(T)| - \chi + 1. \]
\end{proof}

Let us consider the case $\chi=2$ (that is, $G$ is a 3-connected plane graph) in Lemma~\ref{lem2}. 
In this case, as a graph $T$, we take a graph obtained from a spanning tree of $G$ by deleting an edge. 
Then, the graph $T'$ constructed as in Lemma~\ref{lem2} is a connected unicyclic graph. 
By repeating the same argument, we can construct an edge-assignment $\tau$ from $T'$. 
Therefore, we obtain the following lemma (see also Lemma 2 in \cite{AHP2009}). 

\begin{lem}\label{lem3}
Let $G$ be a 3-connected plane graph. 
Let $T$ be the graph obtained from a spanning tree of $G$ by deleting an edge. 
Then, there exists an edge-assignment $\tau$ from $V(G^*)$ to $E(G)\setminus E(T)$ such that $H_{\tau} = T$.
\end{lem}

Let $k$ be a positive integer. 
For a graph $G$ and a spanning subgraph $F$ of $G$, we define the \emph{total excess} $\mathsf{te}(F,k)$ of $F$ from $k$ by
\[ \mathsf{te}(F,k) := \sum_{v\in V(G)} \max \{d_{F}(v)-k, 0\}.\]
If ($G$ is connected and) $F$ is a spanning tree of $G$ with maximum degree at most $k$, $F$ is called a \emph{$k$-tree}. 
By definition, if $F$ is a $k$-tree of $G$, then $\mathsf{te}(F,k)=0$. 

We need the following theorems to complete the proof of Theorem~\ref{thm6}. 

\begin{theorem} [\cite{B1966}]\label{thm7}
Every 3-connected planar graph has a 3-tree.
\end{theorem}

\begin{theorem} [\cite{GR1994}]\label{thm8}
Every 3-connected projective planar graph has a 3-tree.
\end{theorem}

\begin{theorem} [\cite{BEGMR1995}]\label{thm9}
Every 3-connected graph embedded on the torus or Klein bottle has a 3-tree.
\end{theorem}

\begin{theorem} [\cite{O2013}]\label{thm10}
Let $G$ be a 3-connected graph on a surface with Euler characteristic $\chi<0$.
Then $G$ has a spanning tree $T$ with $\mathsf{te}(T,3)\leq -2\chi-1$.
\end{theorem}

\begin{proof} [Proof of Theorem~\ref{thm6}]
Let $G$ be a polyhedral embedding on a closed surface with Euler characteristic $\chi$. 
If $\chi=2$, then let $T$ be the graph obtained from a 3-tree of $G$ in Theorem~\ref{thm7} by deleting an edge. 
Otherwise, let $T$ be a spanning tree of $G$ in Theorems~\ref{thm8},~\ref{thm9}, or~\ref{thm10}. 
We take $H_{\tau}$ as in Lemmas~\ref{lem2} or~\ref{lem3} for $T$ and $\tau$. 
Then, by Lemmas~\ref{lem2} and~\ref{lem3}, we have 
\begin{align*}
\mathsf{te}(H_{\tau},3) &\leq \mathsf{te}(T,3) + 2\ (|E(H_{\tau})|-|E(T)|)\\
&= \max \{ -2\chi-1, 0 \} + 2\ \max \{ -\chi+1, 0 \}\\
&= \begin{cases}
0 & \text{if}\ \chi\in \{1,2\}, \\
2 & \text{if}\ \chi=0, \\
-4\chi + 1 & \text{if}\ \chi<0.
\end{cases}
\end{align*}
Thus, by Propositions~\ref{prop1} and~\ref{prop2}, $G^*$ admits a $G$-decomposition $(G,(G^*_v)_{v\in V(G)})$ with 
\begin{equation*} 
\sum_{v\in V(G)} \max \{|G^*_v|-3, 0\} = \mathsf{te}(H_{\tau},3) \leq \begin{cases}
0 & \text{if}\ \chi\in \{1,2\}, \\
2 & \text{if}\ \chi=0, \\
-4\chi + 1 & \text{if}\ \chi<0.
\end{cases}
\end{equation*}
For each $v\in V(G)$ with $|G^*_v|\geq 4$, choose arbitrarily $|G^*_v|-3$ vertices from $G^*_v$. 
Then the set consisting of these vertices is a desired $S$. 
\end{proof}

\section{Proof of Theorem~\ref{thm5}}\label{sec3}

In this section, we consider the face subdivision of a polyhedral embedding. 
The following theorem is key to the proof of Theorem~\ref{thm5}.

\begin{theorem}\label{thm11}
Let $G$ be a polyhedral embedding on a closed surface with Euler characteristic $\chi$, and let $H$ be the union of $G^{\mathsf{fs}}$ and $(G^*)^{\mathsf{fs}}$. 
Then the following statements hold.
\begin{enumerate}
\item[(1)] If $\chi\in \{1,2\}$, then $G \textnormal{-} \mathsf{wd}(H)\leq 4$. 
\item[(2)] If $\chi=0$, then there exists $S\subseteq V(H)$ with $|S|\leq 2$ such that $G \textnormal{-} \mathsf{wd}(H-S)\leq 4$.
\item[(3)] If $\chi<0$, then there exists $S\subseteq V(H)$ with $|S|\leq -4\chi+1$ such that $G \textnormal{-} \mathsf{wd}(H-S)\leq 4$.
\end{enumerate}
\end{theorem}

By the same argument as in Section~\ref{sec2}, we can see that Theorem~\ref{thm11} implies Theorem~\ref{thm5}. 
Note that $G \textnormal{-} \mathsf{width}$ is non-increasing under taking subgraphs. 
(More strongly, $G \textnormal{-} \mathsf{width}$ is \emph{minor-monotone}, as are treewidth and pathwidth.) 

\begin{proof} [Proof of Theorem~\ref{thm11}]
Let $G$ be a polyhedral embedding on a closed surface with Euler characteristic $\chi$. 
Let $H$ be the union of $G^{\mathsf{fs}}$ and $(G^*)^{\mathsf{fs}}$. 
In Section~\ref{sec2}, we constructed an edge-assignment of $G$ $\tau$ and a $G$-decomposition of $G^*$ $(G, (G^*_v)_{v\in V(G)})$, where $G^*_v := \{ f\in V(G^*) \mid v \text{ is on the face $f$ but not an endpoint of } \tau(f) \}$, which satisfies the following properties. 

\begin{enumerate}
\item[(1')] If $\chi\in \{1,2\}$, then $|G^*_v| \leq 3$ for every $v\in V(G)$. 
\item[(2')] If $\chi=0$, then there exists $S\subseteq V(G^*)$ with $|S|\leq 2$ such that $|G^*_v\setminus S| \leq 3$ for every $v\in V(G)$.
\item[(3')] If $\chi<0$, then there exists $S\subseteq V(G^*)$ with $|S|\leq -4\chi+1$ such that $|G^*_v\setminus S| \leq 3$ for every $v\in V(G)$.
\end{enumerate}

For each $v\in V(G)$, we let $H_v:=G^*_v\cup \{ v \}$. 
To prove Theorem~\ref{thm11}, by (1'), (2'), and (3') together with the definition of $H_v$, it suffices to show that $(G, (H_v)_{v\in V(G)})$ is a $G$-decomposition of $H$. 
Clearly, $(G, (H_v)_{v\in V(G)})$ satisfies the conditions (D1) and (D3). 
It remains to verify the condition (D2). 

Let $e$ be an edge of $H$. 
If $e\in E(G^*)$, since $(G,(G^*_v)_{v\in V(G)})$ is a $G$-decomposition of $G^*$, then (D2) holds.
Thus, we may assume that either (a) $e\in E(G)$ or (b) $e\in E(G^{\mathsf{fs}})\setminus E(G)$. 
Note that $E(G^{\mathsf{fs}})\setminus E(G) = E((G^*)^{\mathsf{fs}})\setminus E(G^*)$.
For (a), if $e=vv'\in E(G)$, then $v\in H_v$ and $v'\in H_{v'}$, and hence (D2) holds. 
For (b), let $e=fv$, where $f$ and $v$ are vertices corresponding to a face of $G$ and a vertex on $f$ in $G$, respectively. 
If $v$ is not an endpoint of $\tau(f)$, then $\{f,v\}\subseteq G^*_v\subseteq H_v$. 
Otherwise there is a unique vertex $v'$ adjacent to $v$ on $f$ which is not an endpoint of $\tau(f)$. 
Then $f\in H_{v'}$ and $v\in H_v$, and hence (D2) holds. 
\end{proof}

\section{Concluding remarks}
For the proof of Theorem~\ref{thm4} (and Theorem~\ref{thm6}), the fact that every polyhedral embedding $G$ has a ``near 3-tree" is crucial, as it determines the coefficient of $\mathsf{pw}(G)$ in  Theorem~\ref{thm4}. 
Hence, if $G$ has a Hamiltonian path, we can reduce the coefficient of $\mathsf{pw}(G)$ from 3 to 2. 
The following graph classes are known to have Hamiltonian paths: 4-connected plane graphs \cite{T1956}, 4-connected projective-plane graphs \cite{TY1994}, and 4-connected graphs embedded in the torus \cite{TYZ2005}. 
For these three graph classes, we can similarly reduce the coefficient of $\mathsf{tw}(G)$ or $\mathsf{pw}(G)$ from 4 to 3 in the bounds of Theorem~\ref{thm5}.

\section*{Acknowledgements} 
The author would like to thank Professor Katsuhiro Ota and Professor Kenta Ozeki for their helpful comments. 
The author was supported by JST ERATO Grant Number JPMJER2301, Japan.

\end{document}